\documentclass[12pt,letterpaper]{article}
\usepackage[margin=1in]{geometry}
\usepackage{graphicx}
\usepackage{enumerate}
\usepackage{amsmath}
\usepackage{amsthm}
\usepackage{amsfonts}
\usepackage{mathrsfs}

\newcommand{\R}{\mathbb{R}}

\newcommand{\K}{\kappa}

\newtheorem{thm}{Theorem}

\if@amsmath
\realias\choose\binom
\else
\renewcommand{\choose}[2]{{{#1}\atopwithdelims(){#2}}}
\fi

\newcommand{\mchoose}[2]{%
  \mathchoice%
    {\left(\kern-0.48em\choose{#1}{#2}\kern-0.48em\right)}
    {\left(\kern-0.30em\choose{\smash{#1}}{\smash{#2}}\kern-0.30em\right)}
    {\left(\kern-0.30em\choose{\smash{#1}}{\smash{#2}}\kern-0.30em\right)}
    {\left(\kern-0.30em\choose{\smash{#1}}{\smash{#2}}\kern-0.30em\right)}
}

\newcommand{\sP}{\mathscr{P}}
\newcommand{\sC}{\mathscr{C}}
\newcommand{\sD}{\mathscr{D}}
\newcommand{\sG}{\mathscr{G}}

\newcommand{\lp}{\left(}
\newcommand{\rp}{\right)}

\newcommand{\fa}{\text{ for all }}

\newcommand{\al}{\alpha}

\newcommand{\g}{\gamma}
\newcommand{\tu}{\tau}
\newcommand{\si}{\sigma}
\newcommand{\be}{\beta}

\newcommand{\es}{\emptyset}

\newcommand{\sube}{\subseteq}

\newcommand{\E}{\mathbb{E}}

\usepackage{tikz}

\usepackage{verbatim}
\usetikzlibrary{arrows,shapes}
\usetikzlibrary{matrix,decorations.pathmorphing}
\usetikzlibrary{decorations.markings}

\title{A Combinatorial Interpretation of the Joint Cumulant}
\author{Connor Ahlbach, Jeremy Usatine and Nicholas Pippenger\footnote{Authors' address: Department of Mathematics, Harvey Mudd College, 301 Platt Blvd., Claremont, CA 91711.}}
\date{\today}

\begin{document}

\begin{titlepage}
\maketitle

\begin{abstract} In this paper, we apply the combinatorial proof technique of Description, Involution, Exceptions (DIE) to prove various known identities for the joint cumulant. Consider a set of random variables $S = \{X_1, \ldots, X_n\} $. Motivated by the definition of the joint cumulant, we define $ \sC(S) $ as the set of cyclically arranged partitions of $S$, allowing us to express the joint cumulant of $ S $ as a weighted, alternating sum over $\sC(S)$. We continue to define other combinatorial objects that allow us to rewrite expressions originally in terms of the joint cumulant as weighted sums over the set of these combinatorial objects. Then by constructing weight-preserving, sign-reversing involutions on these objects, we evaluate the original expressions to prove the identities, demonstrating the utility of DIE.

\end{abstract}

\end{titlepage}

\pagenumbering{arabic}

\section{Introduction}

Let $S = \{X_1, X_2, \ldots, X_n\}$ be a set of random variables. Let $ \sP(A) $ refer to the set of partitions of set $ A $ into nonempty subsets, which we call blocks. So, for $ \tau \in \sP(A) $, each $ \be \in \tau $ is a block. Then, the joint cumulant of the set of random variables $ S $ is defined to be
\[
	\K(S) = \K(X_1, X_2, \ldots, X_n) = \sum_{ \tau \in \sP(S) } (-1)^{|\tau| - 1} (|\tau| - 1)! \prod_{\be \in \tau} \E \lp \prod_{X \in \be} X \rp.
\]
The joint cumulant is a measure of how far random variables are from independence. Notice that if $ n = 1 $ or $ n = 2 $, the joint cumulant reduces to the expected value and covariance, respectively:
\[
	\K(X) = \E(X), \qquad \K(X,Y) = \E(XY) - \E(X) \E(Y).
\]   

The identities we prove here are all known results for the joint cumulant. The cumulant was first introduced by Thiele in \cite{thiele} in 1899. Brillinger provides a history of the development of higher order moments and cumulants as well as their applications in \cite{history}. For another interpretation of most of these identities involving Mobius inversion by Speed, see \cite{speed}.

In this paper, we use the combinatorial method of Description, Involution, Exception - DIE - a technique for evaluating alternating sums, to prove properties of the joint cumulant. We learned this technique from Prof. Arthur Benjamin at Harvey Mudd College, who introduced it in \cite{benjamin}. Because the expressions we are evaluating are in terms of joint cumulants, which are in terms of expectation values, we use a version of DIE that proceeds as follows:

\begin{itemize} 

\item D (Description): Describe a combinatorial interpretation of the expression. Define a set $ C $ of combinatorial objects, a sign function $s: C \to \{-1,1\}$, and a weight function $W: C \to \R$ such that the sum can be written as
$$ \sum_{\rho \in C}s(\rho) W(\rho). $$

\item I (Involution): Find a weight-preserving, sign-reversing involution on a subset $ C_0 \sube C $. That is, find $f: C_0 \to C_0 $, such that for any $\rho \in C_0$,
$$ f(f(\rho)) = \rho, \qquad W(f(\rho)) = W(\rho), \qquad s(f(\rho)) = -s(\rho). $$

\item E (Exception): Determine the elements of $ C - C_0 $, which we call exceptions. Then,
$$ \sum_{\rho \in C} s(\rho) W(\rho) = \sum_{\rho \in C - C_0 } s(\rho) W(\rho). $$ 

\end{itemize}

The reason DIE works is simple: If we find such an involution $ f $, then we can break up $ C_0 $ into pairs $ (x, f(x)) $ such that $ s(x) W(x) $ and $ s(f(x)) W(f(x)) $ cancel in the sum for all $ x \in C_0 $, leaving us with the exceptions. Usually we describe the involution first and then give the exceptions.

In order to use DIE to prove properties of the joint cumulant, we introduce a combinatorial interpretation of the joint cumulant. We define $ \sC(A) $ to be the set of cyclically arranged partitions of set $ A $. That is, first partition $ A $ and then arrange the blocks of the partition into a cycle. In general, such an object $ \sigma \in \sC(A) $ looks like:

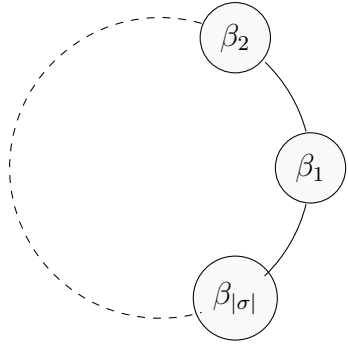
\begin{center}
\begin{tikzpicture}
\tikzstyle{every node} = [draw, shape = circle, fill=gray!5]
\node (1) at (0:2) {$ \beta_1 $} ;
\node (2) at (60:2) {$ \beta_2 $} ;
\node (3) at (300:2) {$ \beta_{|\si|} $} ;

\draw (14:2) arc (15:46:2) ;
\draw (-46:2) arc (-46:-14:2) ;
\draw[style = dashed] (74:2) arc (74:286:2) ;

\end{tikzpicture}
\end{center} 
where the $ \beta_i \in \si $ are the blocks of the partition and $ |\si| $ is the size of the partition, or the number of blocks. Next, we define a weight function for $\sC(S)$. Let $V: \sC(S) \to \R $ be given by
\[
	V(\sigma) = \prod_{ \beta \in \sigma } \E\left( \prod_{X \in \beta} X \right).
\]
Note that $ V $ depends only on the partition, not the cyclic arrangment. For example,

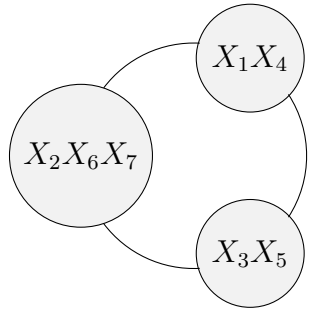
\begin{center}
\begin{tikzpicture}
\tikzstyle{every node} = [draw, shape = circle, fill=gray!10]
\node (1) at (60:1.5) {$ X_1 X_4 $} ;
\node (2) at (180:1.5) {$ X_2 X_6 X_7 $} ;
\node (3) at (300:1.5) {$ X_3 X_5 $} ;

\draw (87:1.5) arc (87:143:1.5) ;
\draw (217:1.5) arc (217:273:1.5) ;
\draw (-33:1.5) arc (-33:33:1.5) ;

\end{tikzpicture}
\end{center}
represents the partition $ \{ \{ X_1,X_4\}, \{X_2,X_6,X_7\},\{X_3,X_5\} \} $ arranged in the cycle \\ $ ( \{ X_1,X_4\} \{X_2,X_6,X_7\} \{X_3,X_5\} ) $, and its weight is 
\[
	\E(X_1 X_4) \E(X_2X_6X_7) \E(X_3X_5). 
\]
Because there are $ (|\tau| - 1)! $ ways to arrange the blocks of $ \tau $ into a cycle, all with value $ \prod_{ \beta \in \tau } \E\left( \prod_{X \in \beta} X \right) $, we readily observe that
\[
	\K(S) = \sum_{ \si \in \sC(S) } (-1)^{ |\si| - 1 } V( \si ).
\]
This is our combinatorial interpretation of the joint cumulant. The set of combinatorial objects is $ \sC(S) $, the sign function is $ s(\si) = (-1)^{ |\si| - 1 } $, and the weight function is $ V(\si) =  \prod_{ \beta \in \si } \E \lp \prod_{X \in \be} X \rp $. We will continue to use the above definition for $ \sC(S) $ and $ V $ throughout this paper. Furthermore, we will define more sets and functions which we will continue to use throught the paper.

\section{Joint Cumulant Identities}

\begin{thm}
\label{SumCoef}
For $ n \ge 2 $, the sum of the coefficients in the joint cumulant is 0. 
\end{thm}

\begin{proof} (DIE)

\begin{itemize} 

\item D (Description): By definition of the joint cumulant, the sum of the coefficients is 
\[
	\sum_{ \tau \in \sP(S) } (-1)^{|\tau| - 1} (|\tau| - 1)! = \sum_{\si \in \sC(S)} (-1)^{ |\si| - 1}.
\]
Here, our set of combinatorial objects is $ \sC(S) $, the sign function is $ s(\si) = (-1)^{ |\si| - 1} $, and the weight function is 1.

\item I (Involution): In any cyclically arranged partition of $ S $, either $ X_1 $ is  or is not alone in its block. If $ X_1 $ is alone, there must exist a block after it as $ n \ge 2 $. Let $ f $ be the map that merges $ X_1 $ into the next block in the cycle if $ X_1 $ is alone and pulls $ X_1 $ out of its block and puts it before that block if $ X_1 $ is not alone. That is, let $ f $ map

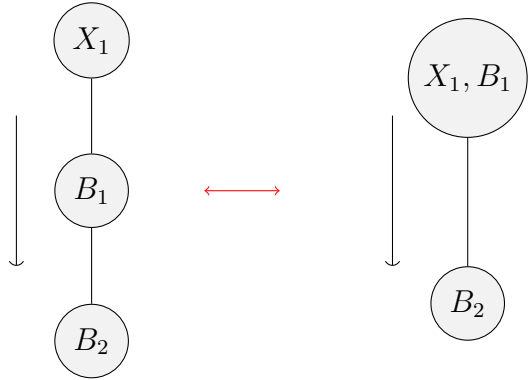
\begin{center}
\begin{tikzpicture}
\tikzstyle{every node} = [draw, shape = circle, fill=gray!10]
\node (1) at (0,2) {$ X_1 $} ;
\node (2) at (0,0) {$ B_1 $} ;
\node (3) at (0, -2) {$ B_2 $} ;

\draw (1)--(2);
\draw (2)--(3);
\draw[-)] (-1,1)--(-1,-1);

\draw[red,<->] (1.5,0) -- (2.5,0);

\node (4) at (5,1.5) {$ X_1, B_1 $} ;
\node (5) at (5,-1.5) {$ B_2 $} ;

\draw (4)--(5);
\draw[-)]  (4,1)--(4,-1);

\end{tikzpicture}
\end{center} 

where $ B_1 $ and $ B_2 $ are nonempty sets of random variables. Clearly, $ f $ is an involution and is weight-preserving as all weights are 1. It is also sign-reversing as it changes the number of blocks by $ \pm 1 $.     

\item E (Exceptions): None. 

\end{itemize}

Therefore, the sum of the coefficients is 0.

\end{proof}

\begin{thm}
\label{(n - 1)-indep}
Suppose $ n \ge 2 $ and every proper subset of $ S $ is independent. Then,
\[
	\K(S) = \E \lp \prod_{ X \in S } X \rp - \prod_{ X \in S } \E(X).  
\] 
\end{thm}

\begin{proof} (DIE)

\begin{itemize} 

\item D (Description): As before,
\[
	\K(S) = \sum_{ \si \in \sC(S) } (-1)^{ |\si| - 1 } V( \si ).
\]

\item I (Involution): Let $ f $ be the same involution as in Theorem \ref{SumCoef}. Notice that each cyclic arrangment of the blocks $ \si \ne \{ S \} $ has weight
\[
	 V( \si) = \prod_{\be \in \si} \E \lp \prod_{X \in \be} X \rp = \prod_{ X \in S } \E(X)  
\]
because every proper subset of $ S $ is independent. Hence, $ f $ is weight-preserving for $ \si \ne \{ S \}, \{ \{X_1\}, \{ X_2, X_3 \cdots X_n \} \} $. It is also sign-reversing as it changes the number of blocks by $ \pm 1 $.     

\item E (Exceptions): We have two exceptions - $ \{ S \} $, with value $ \E \lp  \prod_{ X \in S } X \rp $ and sign $ (1) $, and its pair in the involution, $ \{ \{X_1\}, \{ X_2, X_3 \cdots X_n \} \} $, with value $  \prod_{ X \in S } \E(X) $ and sign $ (-1) $. 
\end{itemize}

Therefore,
\[
	\K(S) = \E \lp \prod_{ X \in S } X \rp - \prod_{ X \in S } \E(X). 
\]

\end{proof}

For the next theorem, we will need the following defintion. We say two sets of random variables $ M = \{ M_1, M_2, \cdots M_p \} $ and $ N = \{ N_1, N_2, \cdots N_q \} $ are independent if
\[
	P( M_i = a_i \fa i , N_j = b_j \fa j) = P( M_i = a_i \fa i ) P( N_j = R_j \fa j).
\]
Furthermore, if $ M $ and $ N $ are independent, then for all $ Q \sube M $, $ R \sube N $, 
\[
	\E \lp \prod_{X \in Q} X \prod_{Y \in R} Y \rp = \E \lp \prod_{X \in Q} X \rp \E \lp \prod_{Y \in R} Y \rp. 
\]

\begin{thm}
If $ S $ can be partitioned into two nonempty, independent subsets, then
\[
	\K(S) = 0.
\] 
\end{thm}

\begin{proof} (DIE)

\begin{itemize} 

\item D (Description): As before,
\[
	\K(S) = \sum_{ \si \in \sC(S) } (-1)^{ |\si| - 1 } V( \si ).
\]

\item I (Involution): Let $ M = \{ M_1, M_2, \cdots M_p \} $ and $ N = \{ N_1, N_2, \cdots N_q \} $ be a partition of $ S $ into two relatively independent subsets. Consider the block containing $ M_1 $. Starting from this block in the cycle, find the first block containing an element of $ N $. Call this block $ B $. Let $ C = B \cap M $ and $ D = B \cap N $. Now, if $ C = \es $, the previous block must have contained only elements of $ M $, or else we would have found the desired block $ B $ sooner, as we started at a block containing $ M_1 $. Then, let $ f $ be the following map: If $ C \ne \es $, pull out $ C $ from $ B $ and make it its own block in front of $ B $. Otherwise, if $ C = \es $, move the previous block into $ B $. That is, let $ f $ map

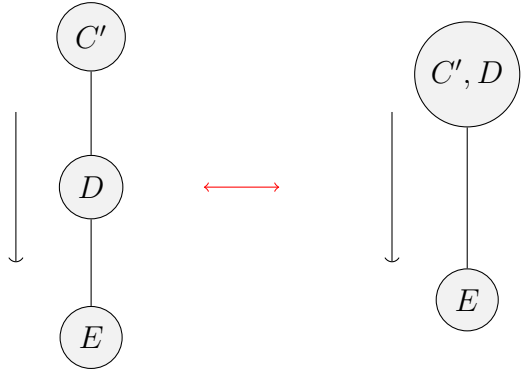
\begin{center}
\begin{tikzpicture}
\tikzstyle{every node} = [draw, shape = circle, fill=gray!10]
\node (1) at (0,2) {$ C' $} ;
\node (2) at (0,0) {$ D $} ;
\node (3) at (0, -2) {$ E $} ;

\draw (1)--(2);
\draw (2)--(3);
\draw[-)] (-1,1)--(-1,-1);

\draw[red,<->] (1.5,0) -- (2.5,0);

\node (4) at (5,1.5) {$ C', D $} ;
\node (5) at (5,-1.5) {$ E $} ;

\draw (4)--(5);
\draw[-)]  (4,1)--(4,-1);

\end{tikzpicture}
\end{center}

where $ C' \sube M, D \sube N $ are nonempty. Clearly, $ f $ is an involution. Because $M$ and $N$ are relatively independent and $ P \subseteq M, Q \subseteq N$, we have
\[
	\E \lp \prod_{X \in Q} X \prod_{Y \in R} Y \rp = \E \lp \prod_{X \in Q} X \rp \E \lp \prod_{Y \in R} Y \rp. 
\]
Also, $ f $ changes no other blocks besides those involving $ C' $ and $ D $. Hence, $ f $ is weight-preserving. It is also sign reversing since it changes the number of blocks by $ \pm 1 $. 
 
\item E (Exceptions): None. 

\end{itemize}

Hence,
\[
	\K(S) = 0.
\]

\end{proof}

To shorten notation, we will make the following definition. For $ \tau \in \sP(S) $, define
\[
	\K( \tau) = \prod_{ \be \in \tau } \K(\be).
\]

\begin{thm}
\label{CumulantSum}
\[
	\E \lp \prod_{ X \in S } X \rp = \sum_{ \tau \in \sP(S) } \K(\tau). 
\]
\end{thm}

\begin{proof} (DIE)

\begin{itemize} 

\item D (Description): Consider the expression 
\[
	\sum_{ \tau \in \sP(S) } \K(\tau) = \sum_{ \tau \in \sP(S) } \prod_{ \be \in \tau } \K( \be) = \sum_{ \tau \in \sP(S) } \prod_{ \be \in \tau  } \sum_{ \si \in \sC(\be) } (-1)^{ |\si| - 1 } V( \si ).
\]
Choose a partition $ \tau$ of $ S $. Then, for each $ \beta \in \tau $, we choose an element of $\sC(\beta)$. Let $\sD(S)$ be the set of these combinatorial objects. For each $ \rho \in \sD(S) $, we refer to each $ \be \in \rho $ as an outer block and each $ \g \in \be $ as an inner block. Such an object looks like 
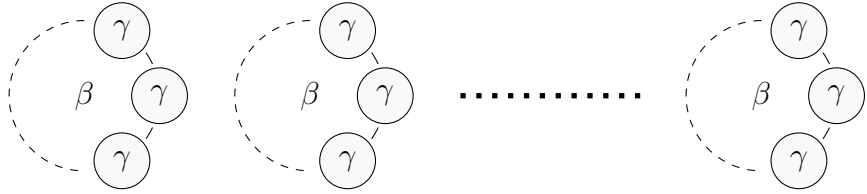
\begin{center}
\begin{tikzpicture}
\path (0,0) node {$\be$} ;
\path (3,0) node {$\be$} ;
\path (9,0) node {$\be$} ;
\tikzstyle{every node} = [draw, shape = circle, fill=gray!5]
\path (0,0) +(0:1) node  {$\g$};
\path (0,0) +(60:1) node {$ \g$};
\path (0,0) +(300:1) node {$ \g$};
\draw[style = dashed] (0,0) +(90:1) arc (90:270:1);
\draw (0,0) +(25:1) arc (25:35:1);
\draw (0,0) +(-35:1) arc (-35:-25:1);

\path (3,0) +(0:1) node {$ \g $};
\path (3,0) +(60:1) node {$ \g $};
\path (3,0) +(300:1) node {$ \g $};
\draw[style = dashed] (3,0) +(90:1) arc (90:270:1);
\draw (3,0) +(25:1) arc (25:35:1);
\draw (3,0) +(-35:1) arc (-35:-25:1);

\draw[style = loosely dotted, line width = 2pt] (5, 0) -- (7.5, 0);

\path (9,0) +(0:1) node {$ \g $};
\path (9,0) +(60:1) node {$ \g $};
\path (9,0) +(300:1) node {$ \g $};
\draw[style = dashed] (9,0) +(90:1) arc (90:270:1);
\draw (9,0) +(25:1) arc (25:35:1);
\draw (9,0) +(-35:1) arc (-35:-25:1);

\end{tikzpicture}
\end{center} 
where the $\g$s are the inner blocks and the $\be$s are the outer blocks. Rewriting the sum using these new objects,
\[
	\sum_{ \tau \in \sP(S) } \K(\tau) = \sum_{ \rho \in \sD(S) } \prod_{ \be \in \rho }  (-1)^{ |\be| - 1 } V( \be ) .	
\]
Or, using our expression for the weight and distributing
\[
	\sum_{ \tau \in \sP(S) } \K(\tau) = \sum_{ \rho \in \sD(S) } (-1)^{ \sum_{ \be \in \rho }|\be| - |\rho|} \prod_{ \be \in \rho } \prod_{ \g \in \be} \E \lp \prod_{ X \in \g } X \rp. 
\]
Thus, our set of combinatorial objects is $ \sD(S) $, our sign function is 
\[
	s(\rho) = (-1)^{ \sum_{ \be \in \rho }|\be| - |\rho|},
\]
and our weight function is
\[
	W(\rho) = \prod_{ \be \in \rho } \prod_{ \g \in \be} \E \lp \prod_{ X \in \g } X \rp. 
\]

\item I (Involution): Consider any $\rho \in \sD(S)$. Let $\g_{1}, \g_2, \ldots, \g_{m}$ be the inner blocks of $\rho$ ordered by increasing order of their minimum index. (This order is arbitrary, but we are explicit about the order because we need the order to be the same for all elements of $\sD(S)$ with the same inner blocks). Then, we write each cycle in terms of the inner blocks using cycle notation. We will keep the inner blocks the same and apply our map to the cycles.

Now, either $ \g_1 $ and $ \g_2 $ do or do not lie in the same outer block or cycle. Let $ f $ be the map that merges the cycles containing $ \g_1 $ and $ \g_2 $ if they are not in the same cycle and splits the cycle containing both if they are in the same cycle. We split and merge these cycles by adding or removing parentheses in cycle notation, respectively. That is, let $ f $ map
\[
	(\g_1 C_1)(\g_2 C_2)(C_3)(C_4) \cdots (C_k) \leftrightarrow (\g_1 C_1 \g_2 C_2) (C_3)(C_4) \cdots (C_k) 
\] 
where each $ C_i $ represents any sequence of inner blocks. Such merging and splitting is well defined because of the order of the cycle. Clearly, $ f $ is an involution. Notice that the weight function is 
\[
	W(\rho) = \prod_{i=1}^m \E\left(\prod_{X \in \g_i} X\right),
\]
which depends only on the inner blocks. As $ f $ does not change the inner blocks, $ f $ is weight-preserving. Furthermore, $ f $ changes the number of outer blocks by $ \pm 1 $ and maintains the number of inner blocks, and we have the sign function $ (-1)^{ \sum_{ \be \in \rho }|\be| - |\rho|} $. Thus, $f$ is sign-reversing.

\item E (Exceptions): This involution is defined for any $\rho$ with two or more inner blocks, so $ f $ is not defined for any $\rho$ with only one inner block, which also means that $\rho$ has one outer block. Thus the only exception is $ \rho = \{ \{ S \} \} $, with sign and weight 
\[
	s(\rho) = 1, \qquad W(\rho) = \E \lp \prod_{ X \in S} X \rp.
\]
\end{itemize}
Therefore,
\[
	\E \lp \prod_{ X \in S } X \rp = \sum_{ \tau \in \sP(S) } \K(\tau). 
\]

\end{proof}

For the next theorem, we will need to define a partial order, $ \preceq $ ,on $ \sP(S) $. For $\pi, \tau \in \sP(S)$, we say that $ \pi $ is finer than $\tau$, or $\pi \preceq \tau$, if for all $\beta \in \pi$, there exists $\alpha \in \tau $ such that $\beta \subseteq \al$. Or, every block of $ \pi $ is contained in a block of $ \tau $.

\begin{thm}
\label{Mobius} For all $ \tau \in \sP(S) $,
\[
	  \prod_{ \beta \in \tau} \E \lp \prod_{ X \in \beta } X \rp = \sum_{ \pi \preceq \tau } \K( \pi ).
\]
\end{thm}

\begin{proof} (DIE)

\begin{itemize} 

\item D (Description): Similarly to Theorem \ref{CumulantSum},
\[
	\sum_{ \pi \preceq \tau } \K( \pi ) = \sum_{ \pi \preceq \tau } \prod_{ \be \in \pi } \sum_{ \si \in \sC(\be) } (-1)^{ |\si| - 1 } V( \si ).
\]
Let $ \sD(S) $ be defined as in Theorem \ref{CumulantSum}. Choose a partition $ \pi $ of $ S $ such that $ \pi \preceq \tau $. Then, for each $ \beta \in \pi $, we choose an element of $\sC(\beta)$. Let $\sD_{\tau}(S) \sube \sD(S) $ be the set of these combinatorial objects. We continue to refer to inner and outer blocks as in Theorem \ref{CumulantSum}. We have the same objects as in \ref{CumulantSum} except the partition forming the outer blocks is finer than $ \tau $. Then let $W$ and $s$ be defined as in Theorem \ref{CumulantSum}. Then, similarly to Theorem \ref{CumulantSum},
\[
	\sum_{ \pi \preceq \tau } \K( \pi ) = \sum_{ \rho \in \sD_\tau(S)} s(\rho) W(\rho).
\]
\item I (Involution): Let $\al_1, \ldots, \al_{|\tau|}$ be the blocks of $ \tau $. Let $\rho \in \sD_\tau(S)$. Now, each inner block is contained in an outer block, and each outer block is contained in a block of $ \tau $ by definition of $ \sD_{ \tau}(S) $. Thus, each inner block is contained in a block of $ \tau $. Let $\g_{i, 1}, \g_{i,2}, \ldots, \g_{i, m}$ be the inner blocks of $\rho$ that are contained in $\al_i$, ordered as in Theorem \ref{CumulantSum}. Find the first $i$ such that there is more than one $\g_{i,j}$. Then let $ f $ be the same map as in Theorem 4, but just acting on the inner blocks with first index $ i $.  Let $ f $ merge the cycles containing $ \g_{i,1} $ and $ \g_{i,2} $ if they are not in the same cycle and split the cycle containing both if they are in the same cycle. That is, let $ f $ map
\[
	(\g_{i,1} C_1)(\g_2 C_2)(C_3)(C_4) \cdots (C_k) \leftrightarrow (\g_{i,2} C_1 \g_2 C_2) (C_3)(C_4) \cdots (C_k) 
\] 
where each $ C_j $ represents any sequence of inner blocks. Clearly, $ f $ is an involution. As before, the weight function depends only on the inner blocks. So, as $ f $ does not change the inner blocks, $ f $ is weight-preserving. Furthermore, $ f $ changes the number of outer blocks by $ \pm 1 $, and we have the same sign function as in Theorem \ref{CumulantSum}, so $ f $ is sign-reversing.

\item E (Exceptions): This involution is defined as long as we have two or more inner blocks for some block of $ \tu $. Thus the exception is the single $ \rho \in \sD_\tau(S) $ for which each outer block contains one inner block, and the inner blocks are the blocks of $\tau$. For this exception,
\[
	s(\rho) = 1, \qquad W(\rho) =  \prod_{ \beta \in \tau} \E \lp \prod_{ X \in \beta } X \rp.
\]
\end{itemize}
Therefore,
\[
	\sum_{ \pi \preceq \tau } \prod_{\beta \in \pi} \K( \beta ) = \prod_{ \beta \in \tau} \E \lp \prod_{ X \in \beta } X \rp.
\]

\end{proof}

For the next theorem, we will need the following definition. Let 
\[
	T = \{ X_{ i,j} \mid i \in [n], j \in [j_i] \}
\]
be a set of random variables, where $[n]$ denotes the set $\{1,2, \ldots, n\}$. Each $ \tau \in \sP([n]) $ induces a partition $ \tau^* \in \sP(T) $ defined by
\[
	\tau^* = \{ \{ X_{i,j} \, | \, i \in \beta, j \in [j_i] \} \, | \, \beta \in \tau \}.
\]
Then we say that $ \pi \in \sP(T) $ is indecomposable if $\pi \preceq \tau^* $ implies $ \tau = \{ [n] \} $. The following result is due to Leonov and Shiryaev in \cite{leonov}.

\begin{thm}
\label{Indecomposable}
\[
	\sum_{ \tau \in \sP(T) }^* \K(\tau) = \K \lp \prod_{ j = 1}^{j_1} X_{1, j},\prod_{ j = 1}^{j_2} X_{2, j}, \ldots, \prod_{ j = 1}^{j_n} X_{n , j} \rp ,
\]
where the $ * $ indicates the sum is over all indecomposable partitions of $ T $.

\end{thm}

\begin{proof} (DIE)

\begin{itemize}

\item D (Description): Let $ \sD(S) $ be defined as in Theorem \ref{CumulantSum}. Choose an indecomposable partition $ \tau $ of $ T $. Then, for each $ \beta \in \tau $, we choose an element of $\sC(\beta)$. Let $\sD^*(T) \sube \sD(T) $ be the set of these combinatorial objects. We continue to refer to inner and outer blocks as in Theorem \ref{CumulantSum}. We have the same objects as in \ref{CumulantSum} except the partition forming the outer blocks is indecomposable. Then let $W$ and $s$ be defined as in Theorem \ref{CumulantSum}. Then, similarly to Theorem \ref{CumulantSum},
\[
	\sum_{ \tau \in \sP(T) }^* \K( \tau ) = \sum_{ \rho \in \sD^*\tau(S)} s(\rho) W(\rho).
\]

\item I (Involution): Consider $\rho \in \sD^*(T)$. We define an $ i $-split as a pair $ \{ X_{i,j}, X_{i,k} \} $ that belong to different inner blocks of $ \rho $. Find the smallest $ i $ such that there exists an $ i $-split, $ \{ X_{i,j}, X_{i,k} \} $, with $j < k$. Let $ \be_1 $ and $ \be_2 $ be the inner blocks containing $ X_{i,j} $ and $ X_{i,k} $, respectively. Now, we perform the same involution as in Theorem \ref{CumulantSum} using these blocks. Let $ f $ merge the cycles containing $ \be_1 $ and $ \be_2 $ if they are not in the same cycle and split the cycle containing both if they are in the same cycle. That is, let $ f $ map
\[
	(\beta_1 C_1)(\beta_2 C_2)(C_3)(C_4) \cdots (C_m) \leftrightarrow (\beta_1 C_1 \beta_2 C_2) (C_3)(C_4) \cdots (C_m).
\] 
where each $ C_r $ represents any sequence of inner blocks. First, we verify that if $ f $ is defined for $ \rho \in \sD^*(T) $, then $ f(\rho) \in \sD^{*}(T) $, or that $ f $ preserves indecomposability of the outer partition of $\rho$. We will consider the cases when $f$ merges cycles and splits cycles separately. 

If $ f $ merges 2 cycles into a single cycle, then the new outer partition is larger under the refinement order. So clearly, indecomposability is preserved in this case.

Otherwise, $ f $ splits a single outer block, $ C $, into 2 outer blocks, say $ C_1 $ and $ C_2 $. Let $ \sigma $ be the outer partition of $ \rho $, and let $ \sigma' $ be the outer partition of $ f(\rho) $. We know $ \sigma $ is indecomposable. Suppose that $ \sigma' \preceq \tau^{*} $ for some $ \tau \in \sP([n]) $. So, $ C_1 $ and $ C_2 $ must be contained in some block of $ \tau^* $. By definition of $ f $, $ C_1 $ and $ C_2 $ both contain an element with first index $ i $. Thus, as $ \tau^* $ is an induced partition, $ C_1 $ and $ C_2 $ must be contained in the same block of $ \tau^{*} $. Hence, $ C $ is also contained in some block of $ \tau^* $. As all other blocks of $ \sigma $ are also blocks of $ \sigma' $, it follows that $ \sigma \preceq \tau^* $ as well. But, because $ \sigma $ is indecomposable, $ \tau = \{ [n] \} $. Thus, $ \sigma' $ is indecomposable as well. Therefore, $ f $ preserves indecomposability.     

As before, $ f $ is a sign-reversing, weight-preserving involution. 

\item E (Exceptions): This involution is defined for any $\rho \in \sD^*(T)$ with an $ i $-split for some $ i $, but it fails if we have no such splits. Note that $\rho \in \sD^*(T)$ cannot have more than one outer block and no splits, because then the outer partition would not be indecomposable. Thus, each exception $\rho \in \sD^*(T)$ has outer partition $ \{ T \} $, and for each index $ i $, all random variables with first index $ i $ must belong to the same inner block. So there exists a unique $\si \in \sC([n])$ for which the set of inner blocks of $\rho$, with the same cyclic order, is
\[
	\{\{ X_{i,j} \, | \, i \in \beta, j \in [j_i] \} \, | \, \beta \in \si\}, 
\]
and an exception of this form exists uniquely for every $\si \in \sC([n])$.The sign and weight of each exception $\rho$ with corresponding $\si$ is
\[
	s(\rho) =  (-1)^{|\si|-1}, \qquad W(\rho) = \prod_{\beta \in \si} \E \left( \prod_{i \in \beta} \left( \prod_{j = 1}^{j_i} X_{i,j} \right) \right).
\]
Hence, the total sum reduces to
\[
	\sum_{\si \in \sC([n])} (-1)^{|\si|-1} \prod_{\beta \in \si} \E \left( \prod_{i \in \beta} \left( \prod_{j = 1}^{j_i} X_{i,j} \right) \right) = \K \lp \prod_{ j = 1}^{j_1} X_{1, j},\prod_{ j = 1}^{j_2} X_{2, j}, \ldots, \prod_{ j = 1}^{j_n} X_{n , j} \rp
\]
by definition of the Joint Cumulant.

\end{itemize}

Therefore,
\[
	\sum_{ \tau \in \sP(T) }^* \K(\tau) = \K \lp \prod_{ j = 1}^{j_1} X_{1, j},\prod_{ j = 1}^{j_2} X_{2, j}, \ldots, \prod_{ j = 1}^{j_n} X_{n , j} \rp.
\]
\end{proof}

Let $ Y $ be a random variable. Define the conditional joint cumulant as
\[
	\kappa( S \, | \, Y ) = \sum_{ \tau \in \sP(S) } (-1)^{|\tau| - 1} (|\tau| - 1)! \prod_{\beta \in \tau}\E \lp \prod_{X \in \beta} X \, \bigg| \, Y \rp
\]
It is the same as the joint cumulant except the expectations are conditional. The following theorem comes from Brillinger in \cite{brillinger}.

\begin{thm}
\label{Conditional}
We can generalize the Tower Formula as follows
\[
	\K(S) = \sum_{ \tau \in \sP(S) } \K(\K(\beta_1 \mid Y), \K(\beta_2 \mid Y ), \ldots, \K(\beta_{|\tau|} \mid Y) ),
\]
where the $\beta_i$ are the blocks of partition $\tau$.
\end{thm}

\begin{proof} (DIE)

\begin{itemize} 

\item D (Description): First, we have	 
\[
	 \sum_{ \tau \in \sP(S) } \K(\K(\beta_1 \mid Y), \K(\beta_2 \mid Y ), \ldots, \K(\beta_{|\tau|} \mid Y) ) 
\]
\[
	= \sum_{ \tau \in \sP(S) } \; \sum_{ \pi \in \sC(  \K(\be_1 \mid Y), \K(\be_2 \mid Y ), \ldots \K(\be_{|\tau|} \mid Y)) } (-1)^{ |\pi| - 1} V(\pi).
\]
Here, each $ \K(\be_i | Y) $ is a sum over all cyclic arrangments of partitions of $ \be_i $. So, the sum represents partitioning $ S $ into blocks which we arrange into inner cycles, and then partitioning the set of cycles into outer blocks, which we arrange into an outer cycle. But this is equivalent to first partitioning $ S $ into outer blocks, which we arrange into a cycle, then partitioning the outer blocks into middle blocks, and then partitioning each middle block into inner blocks, which we arrange into a cycle. 

Choose a cyclically arranged partition $ \si $ of $ S $. Then, for each $ \be \in \si $, choose an element of $ \sD(\be) $. Let $\sG(T) $ be the set of these combinatorial objects. Then, each object $ \rho \in \sG(S) $ looks like

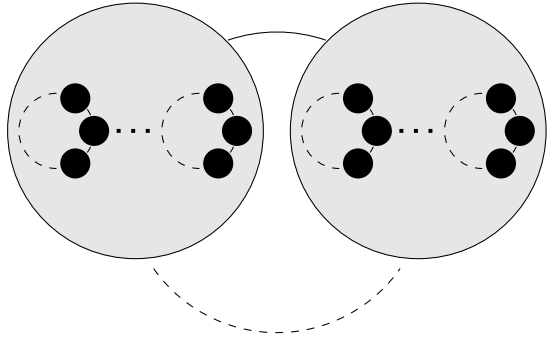
\begin{center}
\begin{tikzpicture}
\path (0,0) +(160: 2) coordinate (a);

\draw[fill=gray!20] (a) circle (1.7);
\path (a) +(-1.05,0) coordinate (a1);
\path[fill=black] (a1) +(0:.5) circle (.2);
\path[fill=black] (a1) +(60:.5) circle (.2);
\path[fill=black] (a1) +(300:.5) circle (.2);
\draw[style = dashed] (a1) +(90:.5) arc (90:270:.5);
\draw (a1) +(25:.5) arc (25:35:.5);
\draw (a1) +(-35:.5) arc (-35:-25:.5);
\draw[style = loosely dotted, line width = 1.5pt] (a1) +(.8,0) -- +(1.3, 0);
\draw (a1) +(1.9,0) coordinate (a2);
\path[fill=black] (a2) +(0:.5) circle (.2);
\path[fill=black] (a2) +(60:.5) circle (.2);
\path[fill=black] (a2) +(300:.5) circle (.2);
\draw[style = dashed] (a2) +(90:.5) arc (90:270:.5);
\draw (a2) +(25:.5) arc (25:35:.5);
\draw (a2) +(-35:.5) arc (-35:-25:.5);

\path (0,0) +(20: 2) coordinate (a);

\draw[fill=gray!20] (a) circle (1.7);
\path (a) +(-1.05,0) coordinate (a1);
\path[fill=black] (a1) +(0:.5) circle (.2);
\path[fill=black] (a1) +(60:.5) circle (.2);
\path[fill=black] (a1) +(300:.5) circle (.2);
\draw[style = dashed] (a1) +(90:.5) arc (90:270:.5);
\draw (a1) +(25:.5) arc (25:35:.5);
\draw (a1) +(-35:.5) arc (-35:-25:.5);
\draw[style = loosely dotted, line width = 1.5pt] (a1) +(.8,0) -- +(1.3, 0);
\draw (a1) +(1.9,0) coordinate (a2);
\path[fill=black] (a2) +(0:.5) circle (.2);
\path[fill=black] (a2) +(60:.5) circle (.2);
\path[fill=black] (a2) +(300:.5) circle (.2);
\draw[style = dashed] (a2) +(90:.5) arc (90:270:.5);
\draw (a2) +(25:.5) arc (25:35:.5);
\draw (a2) +(-35:.5) arc (-35:-25:.5);

\draw (0,0) +(71: 2) arc (71: 109: 2);
\draw[style = dashed] (0,0) +(215: 2) arc (215: 325: 2);

\end{tikzpicture}
\end{center} 
where the black circles are the inner blocks, the cycles of these are the middle blocks, and the larger two gray circles are the outer blocks. For $ \rho \in \sG(S) $, we refer to each $ \al \in \rho $ as an outer block, each $ \be \in \al $ as a middle block, and each $ \g \in \be $ as an inner block. Then, distributing, we see that the sum becomes 
\[
 	\sum_{ \rho \in \sG(S) } (-1)^{|\rho| -1} (-1)^{ \sum_{\alpha \in \rho} \left( \left(\sum_{\beta \in \alpha}|\beta|\right) - |\alpha| \right) } \prod_{\alpha \in \rho} \E\left( \prod_{\beta \in \alpha} \prod_{\gamma \in \beta} \E \left( \prod_{X \in \gamma} X \, \bigg| \, Y \right) \right).
\]
So, our set of combinatorial objects is $ \sG(S) $, our sign function is 
\[
	s(\rho) = (-1)^{|\rho| -1} (-1)^{ \sum_{\alpha \in \rho} \left( \left(\sum_{\beta \in \alpha}|\beta|\right) - |\alpha| \right) },
\]
and our weight function is
\[
	W(\rho) = \prod_{\alpha \in \rho} \E\left( \prod_{\beta \in \alpha} \prod_{\gamma \in \beta} \E \left( \prod_{X \in \gamma} X \, \bigg| \, Y \right) \right).
\]

\item I (Involution): Consider $\rho \in \sG(S)$. Let $\alpha_1, \ldots, \alpha_m$ be the outer blocks, ordered by increasing minimum index. Consider the smallest $ i $ such that $ \alpha_i $ contains more than one inner block. Let $\g_1, \g_2, \ldots, \g_k $ be the inner blocks of $\alpha_i$. Then perform the same merge- split involution as in the previous proofs on these inner blocks. That is, within $\alpha_i$, let $ f $ map
\[
	(\gamma_1 C_1)(\gamma_2 C_2)(C_3)(C_4) \cdots (C_r) \leftrightarrow (\gamma_1 C_1 \gamma_2 C_2) (C_3)(C_4) \cdots (C_r) 
\] 
where each $ C_j $ represents any sequence of inner blocks. As before, $ f $ is an involution. Here, the weight function
\[
	W(\rho) = \prod_{\alpha \in \rho} \E\left( \prod_{\beta \in \alpha} \prod_{\gamma \in \beta} \E \left( \prod_{X \in \gamma} X \, \bigg| \, Y \right) \right)
\]
depends only on the inner and outer blocks. Because $f$ does not change the inner or outer blocks, $ f $ is weight-preserving. Furthermore, $ f $ changes the number of middle blocks by $ \pm 1 $ but does not change the number of inner or outer blocks, and we have the sign function 
\[
	s(\rho) = (-1)^{|\rho| -1} (-1)^{ \sum_{\alpha \in \rho} \left( \left(\sum_{\beta \in \alpha}|\beta|\right) - |\alpha| \right) },
\]
so $f$ is sign-reversing.

\item E (Exceptions): This involution is defined for $\rho \in \sG(S)$ if and only if there exists an outer block containing more than one inner block. Therefore, $\rho$ is an exception if each outer blocks contains only one inner block. So, the set of exceptions has a one-to-one correspondence with $ \sC(S) $, corresponding to choosing the outer blocks and cyclically arranging them. Suppose that the outer blocks of $\rho$, considering also their cyclic order, form the cyclically ordered partition $ \sigma \in \sC(S)$. Then exception $\rho$ has weight and sign
\[
	s(\rho) = (-1)^{ |\sigma| - 1 }, \qquad W(\rho) = \prod_{\beta \in \sigma} \E \left( \E\left( \prod_{X \in \beta} X \, \bigg| \, Y \right) \right) = \prod_{\beta \in \sigma} \E \left( \prod_{X \in \beta} X \right),
\]
by the tower formula. So, the sum totals
\[
	\sum_{\sigma \in \sC(S)} (-1)^{|\sigma| - 1} \prod_{\beta \in \sigma} \E \left( \prod_{X \in \beta} X \right) = \K(S).
\]

\end{itemize}
Therefore,
\[
\sum_{ \tau \in \sP(S) } \K(\K(\beta_1 \mid Y), \K(\beta_2 \mid Y ), \ldots, \K(\beta_{|\tau|} \mid Y) ) = \kappa(S).
\]

\end{proof}


\begin{thebibliography}{MM}

\bibitem[B1]{benjamin}
Benjamin, A. T. and Quinn, J.J. 
``Alternate Approach to Alternating sums: A Method to DIE For'',
{\em College Math.\ Journal.}, 39.3 (2008) 191--201.

\bibitem[B2]{brillinger}
Brillinger, D. R. (1969) ``The Calculation of Cumulant via. Conditioning'' {\em Ann. Inst. Statist. Math} 21, 375-390.

\bibitem[B3]{history}
Brillinger, D. R. (1991) ``Some histroy of the Higher-Order Moments and Spectra'' {\em Statistical Sinica}, 1, 465-476

\bibitem[LS]{leonov}
Leonov, V.P. and Shiryaev, A. N. (1959)
``On a method of calculation of semi-invariants.''
{\em Theor. Prob. Appl.} 4, 319-329.

\bibitem[S]{speed}
Speed, T.P. 
``Cumulants and Partition Lattices''
{\em Austral. J. Statist.}, 25(2) (1983) 378-388.

\bibitem[T]{thiele}
Thiele, T.N. (1899) {\em Om Iagttagelseslaere Halvinvarianter, Overs. Vid. Sels. Forh.} 135-141.


\end{thebibliography}
\end{document}